\documentclass[11pt]{amsart}

\usepackage[latin1]{inputenc}
\usepackage{amsmath,amsfonts,amssymb,mathrsfs,amsthm}
\usepackage{mathrsfs}
\usepackage{xpatch}
\usepackage{color}
\usepackage[pdftex]{hyperref}

\usepackage{graphicx,functan}
\usepackage{framed,fancybox}
\usepackage{bbm}

\newcommand{\R}{\mathbb{R}}
\newcommand{\K}{\mathbb{K}}
\newcommand{\C}{\mathscr{C}}
\newcommand{\bs}{\boldsymbol}
\newcommand{\hsigma}{{H_0^\sigma(\Omega)}}
\newcommand{\kg}{\K_g^\sigma}
\newcommand{\eps}{\varepsilon}
\newcommand{\lraup}{\relbar\joinrel\rightharpoonup}
\newcommand{\bsl}{\bs \langle\!\!\!\bs \langle}
\newcommand{\bsr}{\bs \rangle\!\!\!\bs \rangle}

\newtheorem{theorem}{Theorem}[section]

\newtheorem{remark}[theorem]{Remark}

\newtheorem{proposition}[theorem]{Proposition}
\newtheorem{corollary}[theorem]{Corollary}

\numberwithin{equation}{section}

\begin{document}

	
\title[On a class of nonlocal problems with fractional constraint]{On a class of nonlocal problems with fractional gradient constraint}

\author[A. Azevedo]{Assis Azevedo} 
\address{CMAT and Departamento de Matem\'atica, Escola de C\^encias, Universidade do Minho, Campus de Gualtar, 4710-057 Braga, Portugal}
\email{assis@math.uminho.pt}

\author[J.F. Rodrigues]{Jos\'e Francisco Rodrigues}
\address{CMAFcIO -- Departamento de Matem\'atica, Faculdade de Ci\^encias, Universidade de Lisboa
	P-1749-016 Lisboa, Portugal}
\email{jfrodrigues@ciencias.ulisboa.pt}

\author[L. Santos]{Lisa Santos} 
\address{CMAT and Departamento de Matem\'atica, Escola de Ci\^encias, Universidade do Minho, Campus de Gualtar, 4710-057 Braga, Portugal}
\email{lisa@math.uminho.pt}
	
	
	

	
\begin{abstract}
We consider a Hilbertian and a charges approach to fractional gradient constraint problems of the type $|D^\sigma u|\leq g$, involving the distributional fractional Riesz gradient $D^\sigma$, $0<\sigma <1$, extending previous results on the existence of solutions and Lagrange multipliers of these nonlocal problems.
	
We also prove their convergence as $\sigma\nearrow1$ towards their local counterparts with 
the gradient constraint $|D u|\leq g$.
\end{abstract}

\maketitle
	

\section{Introduction}

Recently the distributional partial derivatives of the Riesz potentials of order $1-\sigma$, $0<\sigma<1$,
\begin{equation*}
\big(D^\sigma u\big)_j=\frac{\partial}{\partial x_j}\left(I_{1-\sigma}u\right)=D_j\left(I_{1-\sigma}u\right),\quad  j=1,\ldots,N,
\end{equation*}
where $I_\alpha$, $0<\alpha<1$, is given by	
\begin{equation*}
I_\alpha u(x)=(I_\alpha*u)(x)=\gamma_{N,\alpha}\int_{\R^N}\frac{u(y)}{|x-y|^{d-\alpha}}\,dy,\qquad\text{with }\gamma_{N,\alpha}=\frac{\Gamma(\frac{N-\alpha}2)}{\pi^\frac{N}2\,2^\alpha\,\Gamma(\frac\alpha2)},
\end{equation*}
as shown to be a useful tool for a fractional vector calculus with the $\sigma$-gradient $D^\sigma$ and $\sigma$-divergence $D^\sigma\!\cdot\,$ (see \cite{ShiehSpector2015}, \cite{ShiehSpector2018}, \cite{ComiStefani2019}, \cite{ComiStefani2020}, \cite{Silhavy2020}). It leads to a new class of fractional partial differential equations and new problems in the calculus of variations  \cite{BellidoCuetoMora2021}. As a consequence of the approximation of the identity by the Riesz kernel as $\alpha\rightarrow 0$ (see \cite{Kurokawa1981}), the $\sigma$-gradient converges to the classical gradient $D$ as $\sigma\nearrow1$, for instance, for smooth functions $u\in \C^\infty_0(\R^N)$ (see also \cite{BellidoCuetoMora2021} and \cite{ComiStefani2019}).
Among the nice properties of $D^\sigma$, in \cite{ShiehSpector2015} it was shown, for $u\in \C^\infty_0(\R^N)$ that
\begin{align}\label{int1}
D^\sigma u & \equiv D\left(I_{1-\sigma} *u\right)=I_{1-\sigma}*D u\\
(-\Delta )^\sigma u& = -D^\sigma\,\cdot\left(D^\sigma u\right)\label{int2}
\end{align}
where $(-\Delta )^\sigma$ is the classical fractional Laplacian in $\R^N$. 

Here we are interested in complementing and extending some results of \cite{RodriguesSantos2019} on elliptic fractional equations of second $\sigma$-order, subjected to a $\sigma$-gradient constraint  
\begin{align}\label{int3}
\left|	D^\sigma u\right|\leq g\quad\text{in $\R^N$}
\end{align}
and have the distributional form
\begin{align}\label{int4}
-	D^\sigma\! \cdot\left(A D^\sigma u+ \Lambda^\sigma\right)=f_\#-D^\sigma\!\cdot \bs f.
\end{align}

We consider the homogeneous Dirichlet problem in a bounded open domain $\Omega\subset\R^N$, with Lipschitz boundary, so that the solution $u$ is to be found in the fractional Sobolev space $\hsigma$, $0<\sigma<1$, and may be extended by zero, belonging to $H^\sigma(\R^N)$. The Lipschitz boundary is sufficient for the $H^\sigma_0(\Omega)$-extension property, which is required in Section \ref{section4}. Although in Sections \ref{section2} and \ref{section3} it is not strictly necessary, we prefer to keep this assumption in order to avoid delicate issues, in particular, with the definition of the classical  space $H^\sigma_0(\Omega)$, which is the natural space to treat the Dirichlet boundary condition.

In \eqref{int4} $A$ is a coercive matrix with bounded variable coefficients (see \eqref{2.1}, \eqref{2.2}) and $f_\#$ and $\bs f$ are given functions making the right hand side an element $f'$ of a suitable dual space.

The vector field $\Lambda^\sigma$ is associated with the constraint \eqref{int3} and may have two possible expressions. As we show in Section \ref{section2}, with a Hilbertian approach, for $g\in L^2(\R^N)$, $g\geq 0$ and $f'\in H^{-\sigma}(\Omega)=\left(\hsigma\right)'$, $\Lambda^\sigma=D^\sigma\gamma$, for a unique $\gamma\in \hsigma$ and it defines an element of the subdifferential of $\kg$, the convex subset of $\hsigma$ of functions satisfying \eqref{int3}. The solution $u$ is then the unique solution to the variational inequality \eqref{2.9} in $\kg$ for the operator $-D^\sigma\cdot \left(AD^\sigma\cdot\right)-f'$.

In the second case, with a strictly positive $g\in L^\infty(\R^N)$ and $f_\#\in L^1(\Omega)$, $\bs f\in\bs L^1(\R^N)=L^1(\R^N)^N$, in Section \ref{section3}, by approximating the unique solution $u$ with a suitable quasilinear penalised Dirichlet problem, we show the existence of at least a generalised nonnegative Lagrange multiplier $\lambda^\sigma\in L^\infty(\R^N)'$, such that $\Lambda^\sigma=\lambda^\sigma D^\sigma u$ and $\lambda^\sigma\left(|D^\sigma u|-g\right)=0$ in the sense of charges, i.e. as an element of $L^\infty(\R^N)'$. 

We recall (see \cite{Yosida1980}), Example 5, Section 9, Ch. IV, for instance, that a charge or an element $\chi\in L^\infty(\mathscr O)'$, in an open set $\mathscr O\subset\R^N$, can be represented by a finitely additive measure $\chi^*$, with bounded total variation, which is also absolutely continuous with respect to the Lebesgue measure and may be given by a Radon integral
\begin{align}\label{int5}
\langle \chi,\varphi\rangle=\int_{\mathscr O}\varphi d\chi^*,\quad\forall \varphi\in L^\infty(\mathscr O).
\end{align}

As a consequence, it is easy to show the H\"older inequality for nonnegative charges $\chi\in L^\infty(\mathscr O)'$ and arbitrary functions $\varphi,\psi\in L^\infty(\mathscr O)$:
\begin{align}\label{int6}
\left|\langle \chi,\varphi\psi\rangle\right|\leq 	\langle \chi,|\varphi|^p\rangle^{\frac{1}{p}}	\ 	\langle \chi,|\psi|^{p'}\rangle^{\frac{1}{p'}},\quad p>1,\ p'=\tfrac{p}{p-1}.
\end{align}

It was proved in \cite{ShiehSpector2015} that, similarly to the classical case $\sigma=1$, the Sobolev, Trudinger and Morrey inequalities also hold for the fractional $D^\sigma$, in particular, there exists a constant $C=C(N,p,\sigma)>0$, such that, for $1<p<\infty$, $\sigma\in (0,1)$,
\begin{align}\label{int7}
\|u\|_{L^q(\R^N)}\leq C\|D^\sigma u\|_{L^p(\R^N)}, \quad u\in \C^\infty_c(\R^N)
\end{align}
where $q=\tfrac{Np}{N-\sigma p}$ if $\sigma<\tfrac{N}{p}$, $q<\infty$ if $\sigma =\tfrac{N}p$ and $q=\infty$ if $\sigma>\tfrac{N}{p}$. In addition, when $\sigma>\tfrac{N}{p}$, we may take in the left hand side of \eqref{int7} the norm of the H\"older continuous functions $\C^\beta_c(\R^N)$, $0<\beta=\sigma-\frac{N}{p}<1$.
As a consequence, we consider $\hsigma$ with the equivalent Hilbertian norm $\|D^\sigma u\|_{L^2(\R^N)}$ (see \cite{ShiehSpector2015}), which is also a consequence of the fractional  Poincar\'e inequality (see \cite{BellidoCuetoMora2021}).

We observe that our results of Sections \ref{section2} and \ref{section3} also hold in the limit local case $\sigma=1$, i.e. in $H^1_0(\Omega)$. We then show in Section \ref{section4}, where we need to work with generalised sequences or nets, that the charges approach to the constrained problem yields the convergence as $\sigma\nearrow1$, of the solution $u^\sigma$ and the generalised Lagrange multiplier $\lambda^\sigma$ to the respective solution $(u,\lambda)\in W^{1,\infty}_0(\Omega)\times L^\infty(\Omega)'$ to the classical problem for $D$. We remark that, in this case, our results are new for data in $L^1$ and the general elliptic operator $-D\cdot(A D\ )$, extending \cite{AzevedoSantos2017} where the charges approach was introduced for $-\Delta$ with $f_\#\in L^2(\Omega)$ and $\bs f=0$. For a recent survey on gradient type constrained problems see \cite{RodriguesSantosSurvey2019}.


\section{The Hilbertian approach with  $\sigma$-gradient constraint \mbox{in $L^2$}}\label{section2}

Let the not necessarily symmetric measurable matrix $A=A(x):\R^N\rightarrow\R^{N\times N}$ satisfy the coercive assumption, for some given $a_*,a^*>0$, 
\begin{equation}\label{2.1}
A(x)\bs \xi\cdot\bs \xi\geq a_*|\bs \xi|^2, \quad \text{a.e. }  x\in\R^N,\ \forall \bs \xi\in\R^N,
\end{equation}
and the boundedness conditions
\begin{equation}\label{2.2}
A(x)\bs \xi\cdot\bs \eta\leq a^*|\bs \xi|\,|\bs \eta|, \quad \text{a.e. }  x\in\R^N, \forall \bs \xi,\ \bs \eta\in\R^N.
\end{equation}

Consider
\begin{equation}\label{2.3}
f_\#\in L^{2^{\#}}(\Omega)\text{ and }	\bs f=(f_1,\ldots,f_N)\in \bs L^2(\R^N),
\end{equation}
where by the Sobolev embeddings \eqref{int7},
$2^\#=\frac{2N}{N+2\sigma}$ if $0<\sigma<\frac{N}{2}$, or $2^\#=q$ for any $q>1$ when $\sigma=\frac{1}{2}$ and $2^\#=1$ when $\frac{1}{2}<\sigma<1$, so that 
\begin{equation}\label{2.4}
\langle f',v\rangle_\sigma=\int_{\Omega} f_{\#}v+\int_{\R^N}\bs f\cdot D^\sigma v,
\end{equation}
for arbitrary $v\in H_0^\sigma(\Omega)$, defines the linear form $f'\in H^{-\sigma}(\Omega)=H_0^\sigma(\Omega)'$, $0<\sigma <1$.  We have 
\begin{equation}\label{2.5}
\exists!\phi\in \hsigma:\ \int_{\R^N} D^\sigma \phi\cdot D^\sigma v=\langle f',v\rangle_\sigma, \quad \forall v\in H_0^\sigma(\Omega).
\end{equation}
The validity of \eqref{2.5} is a consequence of the Fr\'echet-Riesz representation theorem and the choice of the left-hand side of this equality as the inner product in $H^\sigma_0(\Omega)$, as stated in the Introduction.
It follows that $\bs F=D^\sigma\phi\in\bs L^2(\Omega)$ belongs to the image of $\hsigma$ by $D^\sigma$:
\begin{equation}\label{2.6}
\Psi_\sigma=\left\{\bs G\in\bs L^2(\R^N): \bs G=D^\sigma v, v\in\hsigma\right\}=D^\sigma(\hsigma),
\end{equation}
which is a strict Hilbert subspace of $\bs L^2(\R^N)$, for the inner product 
\begin{equation*}
(\bs F,\bs G)_{\Psi_\sigma}=\int_{\R^N} D^\sigma \phi\cdot D^\sigma v,
\end{equation*}
and $\Psi_\sigma$ is isomorphic to $H^{-\sigma}(\Omega)$, by Riesz theorem \eqref{2.5}. Actually this remark extends the well-known case $\sigma=1$, when $D^1$ is the classical gradient $D$.

Consider the nonempty closed convex set
\begin{equation}\label{2.7}
\kg=\left\{v\in\hsigma:|D^\sigma v|\leq g \text{ a.e. in } \R^N\right\},
\end{equation}
where the $\sigma$-gradient threshold $g$ is such that
\begin{equation}\label{2.8}
g\in L^2(\R^N), \qquad g(x)\geq 0\quad \text{ a.e. }x\in \R^N.
\end{equation}

Under the assumption \eqref{2.1} and \eqref{2.2}, $A$ defines a continuous bounded coercive bilinear form over $\hsigma$ and, as an immediate consequence of the Stampacchia theorem (see \cite[p.~95]{Rodrigues1987}, for instance) we have the existence, uniqueness and continuous dependence of the solution $u$, with respect to the linear form \eqref{2.4}, of the following variational inequality
\begin{equation}\label{2.9}
u\in\kg:\int_{\R^N} AD^\sigma u\cdot D^\sigma (v-u)
\geq \int_\Omega f_\#(v-u)+\int_{\R^N}\bs f\cdot D^\sigma(v-u),\ \forall v\in \kg.	
\end{equation}

In particular, if $C_*$ denotes the Sobolev constant, with $L^{2^{*}}(\Omega)=L^{2^{\#}}(\Omega)'$,
\begin{equation*}
\|v\|_{L^{2^*}(\Omega)}\leq C_*\|D^\sigma v\|_{\bs L^2(\R^N)},\quad v\in \hsigma,\quad 0<\sigma\leq 1,
\end{equation*}
and $\widehat{u}$ is the solution corresponding to the data $\widehat{f}_\#, \widehat{\bs f}$, we have
\begin{equation}\label{2.10}
\|u-\widehat{u}\|_{\hsigma}\leq \tfrac{C_*}{a_*}\|f_\#-\widehat{f}_\#\|_{L^{2^{\#}}(\Omega)}+\tfrac{1}{a_*}\|\bs f-\widehat{\bs f}\|_{\bs L^2(\R^N)}.
\end{equation}

It is well-known (see \cite[p.~203]{Lions1969}, for instance) that to solve \eqref{2.9} is equivalent to find $u\in\hsigma$, such that
\begin{equation}\label{2.11}
\Gamma\equiv f'-\mathcal{L}^\sigma_A u\in \partial I_{\kg}(u) \quad \text{in } H^{-\sigma}(\Omega),
\end{equation}
where $\mathcal{L}^\sigma_A:\hsigma \rightarrow H^{-\sigma}(\Omega)$ is the linear continuous operator defined by
\begin{equation*}
\langle \mathcal{L}^\sigma_A w,v\rangle_\sigma=\int_{\R^N}AD^\sigma w\cdot D^\sigma v, \quad \forall v,w\in\hsigma
\end{equation*}
and $\Gamma=\Gamma(u)\in H^{-\sigma}(\Omega)$ is an element of the sub-gradient of the indicatrix function	$I_{\kg}$ of the convex set $\kg$ at $u$:
\begin{equation*}
I_{\kg}(v)=\left\{
\begin{array}{rl}
0 & \text{if }v\in\kg\\
+\infty & \text{if }v\in\hsigma\setminus\kg.
\end{array}\right.
\end{equation*}

By Riesz theorem, there exists a unique $\gamma=\gamma(u)\in \hsigma$ corresponding to $\Gamma=\Gamma(u)$ given by \eqref{2.11} (recall \eqref{2.5}) and the couple $(u,\gamma)\in\kg\times\hsigma$ solves the problem 
\begin{equation}\label{2.12}
\int_{\R^N}(AD^\sigma u+D^\sigma \gamma)\cdot D^\sigma v=\int_\Omega f_\#v+\int_{\R^N}\bs f\cdot D^\sigma v,\quad \forall v\in\hsigma.
\end{equation}

If we denote by $\widehat{\gamma}=\gamma(\widehat{u})$, with $\widehat{u}$ solving \eqref{2.9} with $\widehat{f}_\#$ and $\widehat{\bs f}$ given in \eqref{2.3}, using \eqref{2.10} and \eqref{2.2} we easily obtain, by the Riesz isometry $\|\Gamma\|_{H^{-\sigma}(\Omega)}=\|\gamma\|_{H^\sigma_0(\Omega)}$,
\begin{equation}\label{2.13}
\|\gamma-\widehat{\gamma}\|_{\hsigma}\leq C_*\left(1+\tfrac{a^*}{a_*}\right)\|f_\#-\widehat{f}_\#\|_{L^{2^{\#}}(\Omega)}+\left(1+\tfrac{a^*}{a_*}\right)\|\bs f-\widehat{\bs f}\|_{\bs L^2(\R^N)}.
\end{equation}

We have then proven the following result.
\begin{theorem}\label{t2.1}
Under the previous assumptions, namely \eqref{2.1}, \eqref{2.2}, \eqref{2.3} and \eqref{2.8}, there exists a unique solution of \eqref{2.9}, which also satisfies \eqref{2.12} with a unique $\gamma=\gamma(u)\in \hsigma$, obtained through \eqref{2.11} and depending on the data through \eqref{2.13}.
\end{theorem}

\begin{remark}
This result extends to the Riesz fractional gradient the limit case $\sigma=1$, where the classical  gradient of $u$ and of $\gamma$ are extended by zero in $\R^N\setminus\Omega$. A natural and important question is to find a more direct relation of the potential $\gamma$ with the solution $u$ through the existence of a Lagrange multiplier $\lambda$, such that
\begin{equation}\label{2.14}
D^\sigma \gamma=\lambda D^\sigma u.
\end{equation}
\end{remark}

In the classical case $\sigma=1$, with $A=Id$, $\Omega\subseteq\R^2$ simply connected, and $f'$ and $g$ given by positive constants, corresponding to the elasto-plastic torsion problem, Brézis has proven the existence and uniqueness of a bounded function

\begin{equation*}
\lambda\geq 0\quad \text{ such that }\quad \lambda\left(|Du|-g\right)=0\quad \text{ a.e. in } \Omega,
\end{equation*}
which is even continuous if $\Omega$ is convex (see \cite{RodriguesSantosSurvey2019} for references). Although \eqref{2.14} is an open question in the general case of Theorem \ref{t2.1}, for strictly positive bounded threshold $g$, it has been shown to hold in the sense of finite additive measures in \cite{RodriguesSantos2019}, following the case $\sigma=1$ of \cite{AzevedoSantos2017}. 

Using a variant of a classical penalisation method proposed in \cite[p.~376]{Lions1969} with $\eps\in(0,1)$ and 
\begin{equation}\label{2.16}
k_\eps(t)=0, \ t\leq 0, \quad k_\eps(t)=\tfrac{t}{\eps}, \ 0\leq t\leq\tfrac{1}{\eps},\quad k_\eps(t)=\tfrac{1}{\eps^2},\ t\geq\tfrac{1}{\eps},
\end{equation}
we may consider the approximating quasi-linear problem: find $u_\eps\in\hsigma$, such that 
\begin{equation}\label{2.17}
\int_{\R^N} \left(AD^\sigma u_\eps+\widehat{\kappa}_\eps(u_\eps)\,D^\sigma u_\eps\right)\cdot D^\sigma v
= \int_\Omega f_\#v+\int_{\R^N}\bs f\cdot D^\sigma v,\ \forall v\in \hsigma,	
\end{equation}
where we set 
\begin{equation*}
\widehat{\kappa}_\eps=\widehat{\kappa}_\eps(u_\eps)=k_\eps\left(|D^\sigma u_\eps|^2-g^2\right)\quad \text{  with $k_\eps$ given by \eqref{2.16}}.
\end{equation*}

In the proof of the approximation theorem we shall require the assumption: for each $R>0$ there exists a $g_R$, such that 
\begin{equation}\label{2.18}
g(x)\geq g_R>0, \quad\text{ for a.e. } x\in B_R=\{x\in\R^N:|x|<R\}.
\end{equation}

\begin{theorem}\label{t2.2}
Under the assumptions of Theorem \ref{t2.1}, let also \eqref{2.18} hold. Then the unique solution $u_\eps\in\hsigma$ of \eqref{2.17}, as $\eps\rightarrow 0$ is such that
\begin{align}\label{2.19}
u_\eps\underset{\eps\rightarrow 0}{\lraup}u & \quad\text{ in $\hsigma$-weak}\\
\widehat{\kappa}_\eps D^\sigma u_\eps\underset{\eps\rightarrow 0}{\lraup}D^\sigma \gamma & \quad \text{ in $\Psi'_\sigma$-weak}\label{2.20}
\end{align}
where $(u,\gamma)\in\kg\times\hsigma$ is the unique couple given in Theorem \ref{t2.1} and satisfying \eqref{2.12} and $\Psi_\sigma$ is the vector space defined in \eqref{2.6}.
\end{theorem}
\begin{proof}
Since the quasi-linear operator $\widehat{A}_\eps:\hsigma\rightarrow H^{-\sigma}(\Omega)$ defined by the left hand side of \eqref{2.17} is bounded, strongly monotone, coercive and hemicontinuous, the existence and uniqueness of $u_\eps$ solution to \eqref{2.17} is classical (see \cite{Lions1969}, for instance).
	
Taking $v=u_\eps$ in \eqref{2.17} and recalling that $\widehat{\kappa}_\eps(u_\eps)\geq 0$, it is clear that we have, with $C_\sigma>0$ independent of $\eps$, $0<\eps<1$: 
\begin{equation}\label{2.21}
\|u_\eps\|_{\hsigma}\leq \tfrac{C_*}{a_*}\|f_\#\|_{L^{2^\#}(\Omega)}+\tfrac{1}{a_*}\|\bs f\|_{\bs L^2(\R^N)}\equiv C_\sigma,
\end{equation}
so that we have \eqref{2.19} at least for a generalised subsequence and some $u\in \hsigma$. Consequently, from \eqref{2.17} we also obtain
\begin{equation*}
\|\widehat{\kappa}_\eps D^\sigma u_\eps\|_{\Psi'_\sigma}=
\sup_{\scriptstyle v\in\hsigma\atop \|v\|_\hsigma = 1}\int_{\R^N}\widehat{\kappa}_\eps(u_\eps) D^\sigma u_\eps\cdot D v\leq\left(a_*+a^*\right)C_\sigma,
\end{equation*}
for all $\eps$, $0<\eps<1$, by using \eqref{2.21} and recalling \eqref{2.2}. Here we use the definition \eqref{2.5} and we consider $\bs L^2(\R^N)$, identified to its dual, as a subspace of $\Psi'_\sigma$, the dual of $\Psi_\sigma\subseteq \bs L^2(\R^N)$. Hence, for a generalised subsequence $\eps\rightarrow 0$, we also have 
\begin{equation}\label{2.23}
\widehat{\kappa}_\eps D^\sigma u_\eps \underset{\eps\rightarrow 0}{\lraup}\Lambda\quad \text{ in $\Psi'_\sigma$-weak}.
\end{equation}
	
In order to prove that $u\in\kg$, i.e. $|D^\sigma u|\leq g$ a.e. in $\R^N$, we consider, for $R>0$
\begin{align*}
U_{\eps,R}&=\left\{x\in B_R:0\leq |D^\sigma u_\eps(x)|^2-g^2(x)\leq\sqrt{\eps}\right\} \quad { and }\\ \quad V_{\eps,R}&=\left\{x\in B_R: |D^\sigma u_\eps(x)|^2-g^2(x)>\sqrt{\eps}\right\}
\end{align*}
and, we observe that, using the assumption \eqref{2.18}, \eqref{2.21} and $\widehat\kappa_\eps(|D^\sigma u^\eps|^2-g^2)\ge0$, from \eqref{2.17} it follows
\begin{equation}\label{2.24}
g^2_R\int_{B_R}	\widehat{\kappa}_\eps\leq \int_{\R^N}	\widehat{\kappa}_\eps g^2\leq \int_{\R^N}	\widehat{\kappa}_\eps |D^\sigma u_\eps|^2\leq \tfrac{a_*}2 C_\sigma^2,\quad 0<\eps<1.
\end{equation}
	
Consequently, for all $R>0$, we conclude that $|D^\sigma u|\leq g$ in $B_R$ from 
\begin{align*}
\int_{B_R}\left(|D^\sigma u |-g\right)^+&\le\varliminf_{\eps\rightarrow 0}	\int_{B_R}\left(|D^\sigma   u_\eps|-g\right)^+\\
&=\varliminf_{\eps\rightarrow 0}\left[	\int_{U_{\eps,R}}\left(|D^\sigma   u_\eps|-g\right) +\int_{V_{\eps,R}}\left(|D^\sigma   u_\eps|-g\right)\right]
\end{align*}
since
\begin{equation*}
\int_{U_{\eps,R}}\left(|D^\sigma   u_\eps|-g\right) \leq\frac{1}{g_R}	\int_{U_{\eps,R}}\left(|D^\sigma   u_\eps|^2-g^2\right) \leq \frac{|B_R|\sqrt{\eps}}{g_R}	
\end{equation*}
and
\begin{equation*}
\int_{V_{\eps,R}}\left(|D^\sigma   u_\eps|-g\right)\leq|V_{\eps,R}|^{\frac{1}{2}}\left(\|D^\sigma u_\eps\|_{\bs L^2(B_R) }+\|g\|_{\bs L^2(B_R)} \right)
\leq \left(C_\sigma+   \|g\|_{\bs L^2(\R^N)} \right)|V_{\eps,R}|^{\frac{1}{2}}
\end{equation*}
with
\begin{equation*}
|V_{\eps,R}|=\int_{V_{\eps,R}}1\leq \int_{V_{\eps,R}} \frac{\widehat{\kappa}_\eps}{k_\eps(\sqrt{\eps})}\leq \sqrt{\eps}\int_{B_R} \widehat{\kappa}_\eps\leq \frac{a_*C_\sigma^2}{2g_R^2}\sqrt{\eps}.
\end{equation*}
	
Now, observing that for arbitrary $v\in\kg$ we have
\begin{equation*}
\int_{\R^N}\widehat{\kappa}_\eps D^\sigma u_\eps \cdot D^\sigma (v-u_\eps)\leq \int_{\R^N}\widehat{\kappa}_\eps |D^\sigma u_\eps|\left(|D^\sigma v|-|D^\sigma u_\eps|\right)\leq 0
\end{equation*}
(since $\widehat{\kappa}_\eps>0$ if $|D^\sigma u_\eps|>g\geq |D^\sigma v|$), from \eqref{2.17} we obtain
\begin{equation*}
\int_{\R^N} AD^\sigma u_\eps\cdot D^\sigma (v-u_\eps)\geq \int_\Omega f_\#(v-u_\eps)+\int_{\R^N}\bs f\cdot D^\sigma (v-u_\eps),\ \forall v\in \kg,	
\end{equation*}
and, passing to the limit as $\eps\rightarrow 0$, we conclude that $u$ solves \eqref{2.9}, by using \eqref{2.19} and the lower  semi-continuity
\begin{equation}\label{2.25}
\varliminf_{\eps\rightarrow 0}	\int_{\R^N} AD^\sigma u_\eps\cdot D^\sigma u_\eps\geq \int_{\R^N} AD^\sigma u\cdot D^\sigma u.	
\end{equation}
	
Finally, taking an arbitrary $\bs G=D^\sigma v\in \Psi_\sigma$ and taking $\eps\rightarrow 0$ in \eqref{2.17}, recalling \eqref{2.23}, \eqref{2.12} and \eqref{2.5} we find
\begin{equation*}
\langle \Lambda,\bs G\rangle_{\Psi_\sigma}= \lim_{\eps\rightarrow 0}\int_{\R^N}\widehat{\kappa}_\eps D^\sigma u_\eps\cdot D^\sigma v=
\int_{\R^N}\left(D^\sigma \phi-AD^\sigma u\right)\cdot D^\sigma v =
\int_{\R^N}D^\sigma \gamma\cdot D^\sigma v,
\end{equation*}
yielding the conclusion \eqref{2.20}, by the uniqueness of $u$ and $\gamma$.
\end{proof}

\section{The charges approach with a $\sigma$-gradient constraint \mbox{in $L^\infty$}}\label{section3}

In the framework of the previous section, we consider now the convex set $\kg$ defined by \eqref{2.7} with the assumption
\begin{equation}\label{3.1}
g\in L^\infty(\R^N), \quad 0<g_*\leq g(x)\leq g^*\quad  \text{ a.e. $x$ in } \R^N,
\end{equation}
for some constants $g_*$ and $g^*$. It is clear that $\kg$ is still closed for the topology of $\hsigma$ in the space 
\begin{equation}\label{3.2}
\Upsilon_{\infty}^\sigma(\Omega)=\left\{v\in \hsigma:D^\sigma v\in \bs L^\infty(\R^N)\right\}, \quad 0<\sigma\leq 1,
\end{equation}
and therefore, by the fractional Morrey-Sobolev inequality \eqref{int7} for $\sigma>\frac{N}p$, we have, for all $0<\beta<\sigma$,
\begin{equation}\label{3.3}
\kg\subset 	\Upsilon_{\infty}^\sigma(\Omega)\subset \mathscr{C}^{0,\beta}(\overline\Omega)\subset L^\infty(\Omega).
\end{equation}

Here $ \mathscr{C}^{0,\beta}(\overline\Omega)$ is the space of H\"older continuous functions with exponent $\beta$. As observed in \cite{RodriguesSantos2019},  \eqref{3.3} is a consequence of Theorem 7.63 of \cite{Adams1975} (see also \cite[Th. 2.2]{ShiehSpector2015}), which yields
\begin{equation}\label{3.4}
\|u\|_{L^\infty(\Omega)}\leq C_p\|D^\sigma u\|_{\bs L^p(\R^N)}
\leq C_p\|D^\sigma u\|_{\bs L^\infty(\R^N)}^{1-\frac{2}{p}}
\|D^\sigma u\|_{\bs L^2(\R^N)}^{\frac{2}{p}},\quad \forall u\in \Upsilon_{\infty}^\sigma(\Omega),
\end{equation}
where $C_p>0$ is the Sobolev constant corresponding to any $p> \frac{N}{\sigma}\vee 2$.

Therefore, in this case, we can extend the result of the solvability of the variational inequality \eqref{2.9} with data in $L^1$:
\begin{equation}\label{3.5}
f_\#\in L^1(\Omega)\quad \text{ and } \quad \bs f\in\bs L^1(\R^N).
\end{equation}

\begin{theorem}\label{t3.1}
Under the assumptions \eqref{2.1}, \eqref{2.2}, \eqref{2.3} and \eqref{3.1} the unique solution $u$ to \eqref{2.9} also satisfies the continuous dependence estimates \eqref{2.10}. Moreover, if in addition $(\bs f,f_\#)$ and $(\widehat{\bs f},\widehat{f}_\#)$ also satisfy \eqref{3.5}, the following estimate holds
\begin{equation}\label{3.6}
\|u-\widehat{u}\|_{\hsigma}\leq a_p\|f_\#-\widehat{f}_\#\|_{L^1(\Omega)}^{\frac{1}{2-\frac{2}{p}}}+b_1\|\bs f-\widehat{\bs f}\|_{\bs L^1(\R^N)}^{\frac{1}{2}}.
\end{equation}
where $p> \frac{N}{\sigma}\vee 2$ as in \eqref{3.4} and $a_p,b_1>0$ are constants.

Consequently, the variational inequality \eqref{2.9} is also uniquely solvable with the assumption \eqref{2.3} replaced by \eqref{3.5} and the estimate \eqref{3.6} still holds in this case. 
\end{theorem}
\begin{proof}
While the first part of this theorem is also a direct consequence of Stampacchia theorem, the estimate \eqref{3.6} follows easily from \eqref{2.9}. Indeed, if we set $\overline{u}=u-\widehat{u}$, $\overline{f}_\#=f_\#-\widehat{f}_\#$ and $\overline{\bs f}=\bs f-\widehat{\bs f}$, we have 
\begin{multline}\label{3.7}
a_*\|\overline{u}\|_{\hsigma}^2=a_*\int_{\R^N}\|D^\sigma \overline{u}\|^2\\ \leq 	\|\overline{u}\|_{L^\infty(\Omega)}\|\overline{f}_\#\|_{L^1(\Omega)}+
\|D^\sigma \overline{u}\|_{\bs L^\infty(\Omega)}\|\overline{\bs f}\|_{\bs L^1(\Omega)}\\
\leq C_p\left(2g^*\right)^{1-\frac{2}{p}}
\|D^\sigma \overline{u}\|_{\bs L^2(\Omega)}^{\frac{2}{p}}\|\overline{f}_\#\|_{ L^1(\Omega)}+2g^*\|\overline{\bs f}\|_{\bs L^1(\Omega)},
\end{multline}
by \eqref{3.4} and the assumption \eqref{3.1}. Hence \eqref{3.6} follows easily by applying Young inequality and $\sqrt{\phi+\psi}\leq \sqrt{\phi}+\sqrt{\psi}$ to right hand side of \eqref{3.7} where we obtain the constant $a_p$ and $b_1$ depending on $C_p$, $a_*$, $g^*$ and $p> \frac{N}{\sigma}\vee 2$. The solvability of \eqref{2.9} under the assumption \eqref{3.5} can be easily obtained using \eqref{3.6}, approximating the solution by a Cauchy sequence in $\hsigma$ of solutions 
$u_\nu\underset{\nu\rightarrow 0}{\longrightarrow}u$, where $u_\nu$ solves \eqref{2.9} with approximating sequences
\begin{equation}\label{3.8}
{f_\#}_\nu\underset{\nu\rightarrow 0}{\longrightarrow}f_\#\quad \text{ in }\ L^1(\Omega)\quad \text{ and }\quad
\bs f _\nu\underset{\nu\rightarrow 0}{\longrightarrow}\bs f\quad \text{ in }\ L^1(\R^N) 
\end{equation}
with ${f_\#}_\nu\in L^2(\Omega)$ and $\bs f _\nu\in\bs L^2(\R^N)$, for instance, with $f_\nu=\left(f\wedge\frac{1}{\nu}\right)\vee\left(-\frac{1}{\nu}\right)$ by truncation.
\end{proof}

\begin{remark}\label{r3.2}
This result with $L^1$-data extends Theorem 2.1 of \cite{RodriguesSantos2019} which considered only the case $\bs f\equiv 0$. If the data $f_\#\in L^{2^\#}(\Omega)$ and $\bs f\in \bs L^2(\R^N)\cap \bs L^1(\R^N)$ our approximation Theorem \ref{t2.2} also holds for the solution $(u,\gamma)$ to \eqref{2.11}-\eqref{2.12} under the assumption \eqref{3.1}, which implies $g\in L^2(B_R)$ for all $R>0$, since the proof is the same.
\end{remark}
It is also possible to obtain with $L^1$-data the $\frac{1}{2}$-H\"older continuity of the map $L^\infty(\R^N)\ni g\mapsto u\in \hsigma$ with $g$ satisfying \eqref{3.1} and $u$ solution to \eqref{2.9}, extending the Theorem 2.2 of \cite{RodriguesSantos2019}.

\begin{theorem}\label{t3.2}
Under the assumptions \eqref{2.1}, \eqref{2.2} and \eqref{3.5}, let $u$ and $\widehat{u}$ be the solutions to \eqref{2.9} corresponding to $g$ and $\widehat{g}$ satisfying \eqref{3.1}. Then, there exists a constant $C_*>0$, depending on $g_*$ and the data, but independent of the solutions, such that
\begin{equation}\label{3.9}
\|u-\widehat{u}\|_{\hsigma}\leq C_*\|g-\widehat{g}\|_{L^\infty(\R^N)}^\frac{1}{2}.
\end{equation}
\end{theorem}
\begin{proof}
Denote $\delta=\|g-\widehat{g}\|_{L^\infty(\R^N)}$, and take as test functions in \eqref{2.9}, respectively,
\begin{equation*}
w=\frac{g_*}{g_*+\delta}\widehat{u}\in\kg\quad \text{ and }\quad \widehat{w}=\frac{g_*}{g_*+\delta}u\in\K^\sigma_{\widehat{g}}
\end{equation*}
for the variational inequality for $u$ and for $\widehat{u}$.

Observing that 
\begin{equation*}
|u-\widehat{w}|\leq \frac{\delta}{g_*}|u|\quad\text{ and}\quad |D^\sigma(u-\widehat{w})|\leq \frac{\delta}{g_*}|D^\sigma u|
\end{equation*}
and similarly for $\widehat{u}-w$, we obtain \eqref{3.9} from 
\begin{align*}
a_*\|u-\widehat{u}\|^2_{\hsigma}&\leq \int_{\R^N} AD^\sigma (u-\widehat{u})\cdot D^\sigma(u-\widehat{u})\\
&=\int_{\R^N}AD^\sigma u\cdot D^\sigma(u- w)+\int_{\R^N} AD^\sigma u\cdot D^\sigma(w-\widehat u)\\
&\quad+\int_{\R^N}AD^\sigma\widehat u\cdot D^\sigma(\widehat u-\widehat w)+\int_{\R^N}AD^\sigma\widehat u\cdot D^\sigma(\widehat w- u)\\
&\le\int_\Omega f_\#((u-w)+(\widehat u-\widehat w))+\int_{\R^N}\bs f\cdot D^\sigma((u-w)+(\widehat u-\widehat w))\\
&\quad+\tfrac{2\delta}{g_*}\int_{\R^N}\big|AD^\sigma u\cdot D^\sigma\widehat u\big|\\
&=\int_\Omega f_\#((u-\widehat w)+(\widehat u-w))+\int_{\R^N}\bs f\cdot D^\sigma((u-\widehat w)+(\widehat u-w))\\
	&\quad+\tfrac{2\delta}{g_*}\int_{\R^N}\big|AD^\sigma u\cdot D^\sigma\widehat u\big|\\
\hspace{2cm}&\leq \tfrac{2\delta}{g_*}\left(  C_p{g^*}^{1-\frac{2}{p}}\eta_p^{\frac{2}{p}}\|f_\#\|_{L^1(\Omega)} +  g^*\|\bs f\|_{\bs L^1(\R^N)}+a^*\eta_p^2   \right),
\end{align*}
by using \eqref{3.4} and $\eta_p=a_p\|f_\#\|_{L^1(\Omega)}^{\frac{1}{2-\frac{2}{p}}}+b_1\|\bs f\|_{\bs L^2(\R^N)}^\frac{1}{2}$, which is a general upper bound for $\|D^\sigma u\|_{L^2(\R^N)}$ and $\|D^\sigma \widehat{u}\|_{L^2(\R^N)}$, just by taking $v\equiv 0$ in \eqref{2.9} and calculating as in \eqref{3.6}.
\end{proof}

\begin{remark}
This theorem allows to obtain solutions to quasi-variational inequalities of the type \eqref{2.9}, with the solution dependent on the convex sets $\K_{G[u]}^\sigma$ as in \eqref{2.7} with $g=G[u]$, where $G:L^{2^*}(\Omega)\rightarrow L^\infty_{g_*}(\R^N)$, being $L^\infty_{g_*}(\R^N)=\{h\in L^\infty(\R^N):h(x)\ge g_*>0 \text{ a.e. }x\in\R^N\}$, or $G:\C(\overline\Omega)\rightarrow L^\infty_{g_*}(\R^N)$ are continuous and bounded operators, as in Section 4 of \cite{RodriguesSantos2019}, where only the case $f_\#\in L^2(\Omega)$ and $\bs f\equiv0$ was considered. 
\end{remark}

As we observed in Remark \ref{r3.2}, the solution $u$ to the variational inequality with bounded $\sigma$-gradient constraint and data satisfying \eqref{2.3} also solves \eqref{2.12}, but the extra terms involving $\gamma$ can be interpreted with a Lagrange multiplier $\lambda$ in a generalised sense extending the Theorem 3.1 of \cite{RodriguesSantos2019} to $L^1$-data. Here we use the duality in $L^\infty(\R^N)$ and in $\bs L^\infty(\R^N)$ with the notation
\begin{equation}\label{3.10}
\bsl \lambda \bs{\alpha},\bs{\beta}\bsr=\langle \lambda,\bs{\alpha}\cdot \bs{\beta}\rangle,\quad\forall \lambda \in L^\infty(\R^N)'\ \ \forall \bs{\alpha},\bs{\beta}\in\bs L^\infty(\R^N).
\end{equation}

\begin{theorem}\label{t3.3}
Under the assumptions \eqref{2.1}, \eqref{2.2}, \eqref{3.1} and \eqref{2.3} or \eqref{3.5} there exists $(u,\lambda)\in\Upsilon_{\infty}^\sigma(\Omega)\times 	L^\infty(\R^N)'$, such that 
\begin{multline}\label{3.11}
\int_{\R^N}AD^\sigma u\cdot D^\sigma w+\bsl\lambda D^\sigma u, D^\sigma w\bsr\\
=\int_\Omega f_\#w+\int_{\R^N}\bs f\cdot D^\sigma w,\quad \forall w\in\Upsilon^\sigma_\infty(\Omega),
\end{multline}
\begin{equation}\label{3.12}
|D^\sigma u|\leq g\ \text{ a.e. in }\ \R^N,\qquad \lambda\geq 0 \ \text{ and }\ \lambda(|D^\sigma u|-g)=0\ \text{ in }\ L^\infty(\R^N)'.
\end{equation}

Moreover, $u$ is the unique solution to the variational inequality \eqref{2.9}.
\end{theorem}
\begin{proof} {\bf i)}
First we suppose \eqref{2.3}, i.e. $f_\#\in L^2(\Omega)$ and $\bs f\in\bs L^2(\R^N)$, and from the approximation problem \eqref{2.17}, in addition to \eqref{2.21}, we obtain the {\em a priori} estimates independent of $0<\eps<1$:
\begin{align}\label{3.13}
\|\widehat{\kappa}_\eps\|_{L^1(\R^N)}\leq \tfrac{a_*}{2g_*^2}C_\sigma^2\equiv \tfrac{C_1}{g_*^2} 
\\ \label{3.14}
\|\widehat{\kappa}_\eps\|_{L^\infty(\R^N)'}\leq \tfrac{C_1}{g_*^2} 
\\ \label{3.15}
\|\widehat{\kappa}_\eps D^\sigma u_\eps\|_{L^\infty(\R^N)'}\leq \tfrac{C_1}{g_*}.
\end{align}

Indeed, \eqref{3.13} follows from \eqref{2.24} with the assumption \eqref{3.1}, which implies \eqref{3.14}, by definition of the dual norm, as well as \eqref{3.15}, by using \eqref{3.13} and again \eqref{2.24}:
\begin{equation*}
\|\widehat{\kappa}_\eps D^\sigma u_\eps\|_{L^\infty(\R^N)'}=
\sup_{\scriptstyle \bs \beta \in\bs L^\infty(\R^N)\atop \|\bs \beta\|_{\bs L^\infty(\R^N)}=1}\int_{\R^N} \widehat{\kappa}_\eps D^\sigma u_\eps\cdot \bs \beta \leq 
\left(\int_{\R^N} \widehat{\kappa}_\eps |D^\sigma u_\eps|^2\right)^{\frac{1}{2}}	\left(\int_{\R^N} \widehat{\kappa}_\eps \right)^{\frac{1}{2}}\leq \tfrac{C_1}{g_*}.		
\end{equation*}

By the estimates \eqref{3.14}, \eqref{3.15} and the Banach Alaoglu Bourbaki theorem, at least for some generalised subsequence $u_\eps\underset{\eps\rightarrow 0}{\lraup}u$ in $\hsigma$ also
\begin{equation*}
\widehat{\kappa}_\eps \underset{\eps\rightarrow 0}{\lraup}\lambda\ \text{ weakly in }\ L^\infty(\R^N)'\ \text{ and }\
\widehat{\kappa}_\eps D^\sigma u_\eps  \underset{\eps\rightarrow 0}{\lraup}\Lambda\ \text{ weakly in }\ \bs L^\infty(\R^N)'.
\end{equation*}

Since $\widehat{\kappa}_\eps\geq 0$ a.e., $\lambda\geq 0$ in $L^\infty(\R^N)'$ and letting $\eps\rightarrow 0$ in \eqref{2.17} with $w\in \Upsilon^\sigma_{\infty}(\Omega)$, $u$ and $\Lambda$ satisfy
\begin{equation}\label{3.17}
\int_{\R^N}AD^\sigma u\cdot D^\sigma w+\bsl\Lambda, D^\sigma w\bsr=\int_\Omega f_\#w+\int_{\R^N}\bs f\cdot D^\sigma w,\quad \forall w\in\Upsilon^\sigma_\infty(\Omega).
\end{equation}

Letting $\eps\rightarrow 0$ in \eqref{2.17} with $v=u_\eps$ and using \eqref{2.25} we easily find
\begin{equation*}
\varlimsup_{\eps\rightarrow 0}\int_{\R^N} \widehat{\kappa}_\eps |D^\sigma u_\eps|^2\leq \bsl \Lambda, D^\sigma u\bsr.
\end{equation*}

Recalling that $(|D^\sigma u_\eps|^2-g^2)\widehat{\kappa}_\eps\geq 0$ and $|D^\sigma u|\leq g$ a.e $x\in\R^N$, we obtain
\begin{equation*}
\langle \lambda, |D^\sigma u|^2\rangle\leq \langle \lambda,g^2\rangle=	\lim_{\eps\rightarrow 0}\int_{\R^N} \widehat{\kappa}_\eps g^2\leq 
\varlimsup_{\eps\rightarrow 0}\int_{\R^N} \widehat{\kappa}_\eps |D^\sigma u_\eps|^2\leq \bsl \Lambda, D^\sigma u\bsr.
\end{equation*}

Since we get the opposite inequality from 
\begin{multline*}
0\leq \varlimsup_{\eps\rightarrow0}\int_{\R^N} \widehat{\kappa}_\eps |D^\sigma (u_\eps-u)|^2\\ =
\varlimsup_{\eps\rightarrow 0}\int_{\R^N} \widehat{\kappa}_\eps |D^\sigma u_\eps|^2-
2 \lim_{\eps\rightarrow 0}\int_{\R^N} \widehat{\kappa}_\eps D^\sigma u_\eps\cdot D^\sigma u+
\lim_{\eps\rightarrow 0}\int_{\R^N} \widehat{\kappa}_\eps |D^\sigma u|^2\\
\leq \bsl\Lambda,D^\sigma u\bsr-2 \bsl\Lambda,D^\sigma u\bsr+\langle\lambda,|D^\sigma u|^2\rangle=
-\ \bsl\Lambda,D^\sigma u\bsr+\langle\lambda,|D^\sigma u|^2\rangle
\end{multline*}
we conclude $\bsl\Lambda,D^\sigma u\bsr=\langle\lambda,|D^\sigma u|^2\rangle$ and 
\begin{equation}\label{3.19}
\lim_{\eps\rightarrow 0}\int_{\R^N} \widehat{\kappa}_\eps |D^\sigma (u_\eps-u)|^2=0.
\end{equation}

Hence for any $\bs \beta\in \bs L^\infty(\R^N)$, we have
\begin{multline*}
\left|\bsl\Lambda-\lambda D^\sigma u,\bs\beta\bsr\right|=
\lim_{\eps\rightarrow 0}\left|\int_{\R^N} \widehat{\kappa}_\eps D^\sigma (u_\eps-u)\cdot\bs\beta\right|\\
\leq\lim_{\eps\rightarrow 0}\left[\left(\int_{\R^N} \widehat{\kappa}_\eps |D^\sigma (u_\eps-u)|^2\right)^{\frac{1}{2}}\| \widehat{\kappa}_\eps\|_{L^1(\R^N)}\,\|\bs\beta\|_{\bs L^\infty(\R^N)}\right]=0,
\end{multline*}
showing that 
\begin{equation*}
\Lambda=\lambda D^\sigma u\quad \text{ in }\ \bs L^\infty(\R^N)',
\end{equation*}
and that, in fact, \eqref{3.17} is equivalent to \eqref{3.11}.

It remains to show the last equation of \eqref{3.12} which follows easily from (recall \eqref{3.1})
\begin{multline*}
0=\langle\lambda,(g^2-|D^\sigma u|^2)\varphi\rangle=
\langle\lambda,(g-|D^\sigma u|)(g+|D^\sigma u|)\varphi\rangle\\
\geq g_*\langle\lambda,(g-|D^\sigma u|)\varphi\rangle=g_*\langle\lambda(g-|D^\sigma u|),\varphi\rangle\geq 0
\end{multline*}
for arbitrarily $\varphi\in L^\infty(\Omega)$, $\varphi\geq 0$, which holds provided we show  
\begin{equation}\label{3.22extra}
\langle\lambda,(g^2-|D^\sigma u|^2)\varphi\rangle=0.
\end{equation}

As above, using \eqref{3.19}, we have first
\begin{align*}
\langle\lambda,g^2\varphi\rangle& \leq \lim_{\eps\rightarrow 0}\int_{\R^N} \widehat{\kappa}_\eps |D^\sigma u_\eps|^2\varphi\\
&=\lim_{\eps\rightarrow 0}\left(\int_{\R^N} \widehat{\kappa}_\eps |D^\sigma (u_\eps-u)|^2\varphi+2\int_{\R^N} \widehat{\kappa}_\eps D^\sigma (u_\eps-u)\cdot D^\sigma u\, \varphi \right.\\
&\left.\qquad+\int_{\R^N} \widehat{\kappa}_\eps |D^\sigma u|^2\varphi\right)\\
&=\langle \lambda,|D^\sigma u|^2\varphi\rangle
\end{align*}
and, since $u\in\kg$ and $\varphi, \lambda\geq0$, it also holds
\begin{equation*}
\langle\lambda,(g^2-|D^\sigma u|^2)\varphi\rangle\ge0.
\end{equation*}

To show that $u$ is the unique solution to \eqref{2.9} it suffices to take $w=u-v$, with an arbitrary $v\in\kg$, and observe that, by \eqref{3.22extra},
\begin{multline*}
\bsl \lambda D^\sigma u,D^\sigma (v-u)\bsr\leq \langle\lambda, |D^\sigma u|\left(|D^\sigma v|-|D^\sigma u|\right)\rangle\\
\leq \langle\lambda,|D^\sigma u|(g-|D^\sigma u|)\rangle=\langle\lambda(g^2-|D^\sigma u|^2),\tfrac{|D^\sigma u|}{g+|D^\sigma u|}\rangle= 0.
\end{multline*}

{\bf ii)} In the second case, if \eqref{3.5} holds, we can use approximation by solutions $(u_\nu,\lambda_\nu)$ of \eqref{3.11}-\eqref{3.12} corresponding to data ${f_\#}_\nu\in L^{2^\#}(\Omega)$ and $\bs f_\nu\in\bs L^2(\R^N)$ satisfying \eqref{3.8}, as in Theorem \ref{t3.1}.

Using the estimate \eqref{3.6} it is clear that 
\begin{equation}\label{3.22}
u_\nu\underset{\nu\rightarrow 0}{\longrightarrow}u \quad\text { in }\quad \hsigma
\end{equation}
and $u$ solves \eqref{2.9}.

For $\varphi\in L^\infty(\R^N)$, setting $b=\frac{\|\varphi\|_{L^\infty(\R^N)}}{g_*^2}$,  recalling \eqref{3.1} and using \eqref{3.11} and \eqref{3.12} for $\lambda_\nu$, which also implies $\langle \lambda_\nu,g^2-|D^\sigma u_\nu|^2\rangle=0$, we have	
\begin{align}\label{3.24extra}
\nonumber|\langle \lambda_\nu,\varphi\rangle|&\leq \langle \lambda_\nu,bg^2\rangle\\
&=b\langle \lambda_\nu,|D^\sigma u_\nu|^2\rangle=b\bsl \lambda_\nu D^\sigma u_\nu,D^\sigma u_\nu\bsr\\
\nonumber&\leq b\left(\int_\Omega f_\#u_\nu+\int_{\R^N}\bs f\cdot D^\sigma u_\nu\right)\leq C\tfrac{\|\varphi\|_{L^\infty(\R^N)}}{g_*^2},
\end{align}
where the constant $C>0$ depends only on the $L^1$-norms of $f_\#$ and $\bs f$ and on the constants $a_p$ and $b_1$ of \eqref{3.6}, being consequently independent of $\nu$. Then $\lambda_\nu$ is uniformly bounded in $L^\infty(\R^N)'$ and we may assume, for some generalised subsequence,
\begin{equation}\label{3.23}
\lambda_\nu  \underset{\nu\rightarrow 0}{\lraup}\lambda\quad  \text{ in }\  L^\infty(\R^N)'-\text{weakly}^*,\quad \text{ with }\quad \lambda\geq 0,
\end{equation}
and, since $\Lambda_\nu=\lambda_\nu D^\sigma u_\nu$ is also bounded in $\bs L^\infty(\R^N)'$ (recall $\|D^\sigma u_\nu\|_{\bs L^\infty(\R^N)}\leq g^*$), also
\begin{equation}\label{3.24}
\Lambda_\nu  \underset{\nu\rightarrow 0}{\lraup}\Lambda\quad \text{ in }\  \bs L^\infty(\R^N)'-\text{weakly}^*.
\end{equation}

Therefore, taking the limit $\nu\rightarrow 0$ in \eqref{3.11}  we find that $(u,\lambda)$ solves
\begin{equation}\label{3.25}
\int_{\R^N}AD^\sigma u\cdot D^\sigma w+\bsl\Lambda, D^\sigma w\bsr=\int_\Omega f_\#w+\int_{\R^N}\bs f\cdot D^\sigma w,\quad \forall w\in\Upsilon_\infty^\sigma(\Omega).
\end{equation}

Recalling \eqref{3.22extra} with $\varphi=1$, we have
\begin{equation}\label{extra}
\langle\lambda_\nu,|D^\sigma u|^2\rangle\le\langle\lambda_\nu,g^2\rangle=\langle\lambda_\nu,|D^\sigma u_\nu|^2\rangle.
\end{equation}
%
Using the equality \eqref{extra} and \eqref{3.22}, we have
\begin{align}\label{3.27}
\nonumber0\leq \tfrac{1}{2}\langle&\lambda_\nu,|D^\sigma(u_\nu-u)|^2\rangle\\
=&\tfrac{1}{2}
\left(\langle\lambda_\nu,|D^\sigma u_\nu|^2\rangle
-2\langle\lambda_\nu,D^\sigma u_\nu\cdot D^\sigma u\rangle
+\langle\lambda_\nu,|D^\sigma u|^2\rangle
\right)\\
\nonumber\leq& \langle\lambda_\nu,|D^\sigma u_\nu|^2\rangle
-\langle\lambda_\nu,D^\sigma u_\nu\cdot D^\sigma u\rangle
=\bsl\lambda_\nu D^\sigma u_\nu,D^\sigma(u_\nu-u)\bsr\\
\nonumber=& \int_\Omega {f_\#}_\nu(u_\nu-u)+\int_{\R^N}\bs f_\nu\cdot D^\sigma(u_\nu-u)-\int_{\R^N} AD^\sigma u_\nu\cdot D^\sigma (u_\nu-u)\underset{\nu\rightarrow 0}{\longrightarrow}0,
\end{align}
being the last equality satisfied because $(u_\nu,\lambda_\nu)$ solves problem \eqref{3.11}-\eqref{3.12} with data ${f_\#}_\nu$ and $\bs f_\nu$.

Then, from \eqref{3.25} we can conclude that $u$ in fact solves \eqref{3.11} from the equality
\begin{align}\label{3.28}
\nonumber\bsl\Lambda,D^\sigma w\bsr&=\lim_{\nu\rightarrow0}\bsl\lambda_\nu D^\sigma u_\nu,D^\sigma w\bsr\\
&=\lim_{\nu\rightarrow0}\bsl\lambda_\nu D^\sigma u,D^\sigma w\bsr+\lim_{\nu\rightarrow0}\bsl\lambda_\nu D^\sigma (u_\nu-u),D^\sigma w\bsr\\
\nonumber&=\lim_{\nu\rightarrow0}\langle\lambda_\nu,D^\sigma u\cdot D^\sigma w\rangle=\langle\lambda,D^\sigma u\cdot D^\sigma w\rangle=\bsl\lambda D^\sigma u,D^\sigma w\bsr,
\end{align}
which is valid for all $w\in\Upsilon^\sigma_\infty(\Omega)$, since \eqref{3.27} implies
\begin{multline*}
\left|	\bsl\lambda_\nu D^\sigma (u_\nu-u),D^\sigma w\bsr\right|=\left|
\langle\lambda_\nu,D^\sigma (u_\nu-u)\cdot D^\sigma w\rangle\right|\\\leq\langle\lambda_\nu,|D^\sigma (u_\nu-u)|\,|D^\sigma w|\rangle\\
\leq\left(\langle \lambda_\nu,|D^\sigma (u_\nu-u)|^2\rangle\right)^\frac{1}{2}\left(\langle \lambda_\nu,|D^\sigma w|^2\rangle\right)^\frac{1}{2}\underset{\nu\rightarrow 0}{\longrightarrow} 0,
\end{multline*}
where we have used the H\"older inequality for charges  in the last inequality.

From \eqref{3.28}, we find $\bsl\Lambda, D^\sigma u\bsr=\langle \lambda,|D^\sigma u|^2\rangle$ and
\begin{multline*}
\langle\lambda,g^2\rangle=\lim_{\nu\rightarrow0}\langle\lambda_\nu,g^2\rangle=\lim_{\nu\rightarrow0}\bsl\lambda_\nu D^\sigma u_\nu,D^\sigma u_\nu\bsr\\
=\lim_{\nu\rightarrow0}\bsl\lambda_\nu D^\sigma u_\nu,D^\sigma u\bsr+\lim_{\nu\rightarrow0}\bsl\lambda_\nu D^\sigma u_\nu,D^\sigma (u_\nu-u)\bsr\\
=\lim_{\nu\rightarrow0}\bsl\Lambda_\nu,D^\sigma u\bsr=\bsl\Lambda,D^\sigma u\bsr=\langle\lambda,|D^\sigma u|^2\rangle.
\end{multline*}

Finally, we can now complete the proof of the theorem by using this equality in the form 
$\langle\lambda(g^2-|D^\sigma u|^2),1\rangle=0$ and again the H\"older inequality to conclude the third condition in \eqref{3.12} with an arbitrarily $\varphi\in L^\infty(\R^N)$, 
\begin{align*}
\big|\langle\lambda(g-|D^\sigma u|),\varphi\rangle \big|&\leq \langle\lambda(g-|D^\sigma u|),|\varphi|\rangle\\
&=\big
\langle\lambda(g^2-|D^\sigma u|^2),\tfrac{|\varphi|}{g+|D^\sigma u|}\big\rangle\\
&\leq \langle\lambda(g^2-|D^\sigma u|^2),1\rangle^\frac{1}{2}\ \big\langle\lambda(g^2-|D^\sigma u|^2),\tfrac{|\varphi|^2}{(g+|D^\sigma u|)^2}\big\rangle^\frac{1}{2}=0.
\end{align*}
\end{proof}

The second part of this proof actually shows a generalised continuous dependence of the solution and of the Lagrange multiplier with respect to the $L^1$ data.
\begin{corollary}
Under the assumptions \eqref{2.1}, \eqref{2.2}, \eqref{3.1} and \eqref{3.5}, if $(u_\nu,\lambda_\nu)\in\Upsilon^\sigma_\infty(\Omega)\times L^\infty(\R^N)'$ are the solutions to \eqref{3.11}, \eqref{3.12} corresponding to $L^1$ data satisfying \eqref{3.8}, as $\nu\rightarrow 0$, we have the convergence, for some generalised subsequence or net,
\begin{equation*}
u_\nu\underset{\nu\rightarrow 0}{\longrightarrow}u \quad\text { in }\quad \hsigma\quad \text{ and }\quad	\lambda_\nu  \underset{\nu\rightarrow 0}{\lraup}\lambda\quad \text{ in }\  L^\infty(\R^N)'-\text{weakly}^*,
\end{equation*}
where $(u,\lambda) \in\Upsilon^\sigma_\infty(\Omega)\times L^\infty(\R^N)'$ also solves \eqref{3.11}-\eqref{3.12}.
\end{corollary}

\section{Convergence to the local problem as $\sigma\nearrow1$}\label{section4}

It is easy to check that all the theorems of the preceding two sections hold in the limit case $\sigma=1$, when $D^\sigma=D$ is the classical gradient and the data $f_\#$ and $\bs f$ satisfy \eqref{2.3} (with $f_\#\in L^\frac{2N}{N+2}(\Omega)$, if $N>2$, $f_\#\in L^q(\Omega)$, $\forall q<\infty$ if $N=2$ and $q=\infty$ if $N=1$) or \eqref{3.5}, and $g$ satisfies \eqref{2.8}, \eqref{2.18} or \eqref{3.1}, respectively.

In this section we show a continuous dependence of the solution $u^\sigma$ and of the Lagrange multiplier $\lambda^\sigma$ when $\sigma\nearrow1$. For the sake of simplicity, we take $f_\#=0$ and $\bs f\in \bs L^1 (\R^N)$, so that the limit variational inequality reads
\begin{equation}\label{4.1}
u \in\K_g =\big\{v\in H^1_0(\Omega):|Dv|\le g\text{ a.e. in }\Omega\big\},
\end{equation}
\begin{equation}\label{4.2}
\int_\Omega ADu \cdot D(v-u )\ge\int_\Omega \bs f\cdot D(v-u ),\quad\forall v\in\K _g.
\end{equation}

Likewise, observing that setting $\sigma=1$ in \eqref{3.2} we have $\Upsilon_\infty (\Omega)=W^{1,\infty}_0(\Omega)$, we can write the limit Lagrange multiplier problem in the form: find $(u ,\lambda )\in W^{1,\infty}_0(\Omega)\times L^\infty(\Omega)'$
\begin{equation}\label{4.3}
\int_\Omega ADu \cdot Dw+\bsl\lambda Du ,Dw\bsr=\int_\Omega \bs f\cdot Dw,\quad\forall w\in W^{1,\infty}_0(\Omega),
\end{equation}
\begin{equation}\label{4.4}
|Du |\le g\quad\text{ a.e. in }\Omega,\quad \lambda \ge0\ \text{ and }\ \lambda (|Du |-g)=0\quad\text{ in }\ L^\infty(\Omega)'.
\end{equation}

In \eqref{4.3} we denote the duality in $\bs L^\infty(\Omega)$ similarly to \eqref{3.10}, as we can always consider the solution and the test functions extended by zero in $\R^N\setminus\Omega$, since $\partial\Omega$ is $\C^{0,1}$.

We first recall an important consequence of the fact that the Riesz kernel is an approximation of the identity, as remarked by Kurokawa in \cite{Kurokawa1981}.

\begin{proposition}   If $h\in L^p(\R^N)\cap\C(\R^N)$, for some $p\ge 1$, is bounded and uniformly continuous in $\R^N$ then
\begin{equation*}
\lim_{\alpha\rightarrow0}\|I_\alpha*h-h\|_{L^\infty(\R^N)}=0.
	\end{equation*}
	
As a consequence, we have
\begin{equation}\label{4.6}
D^\sigma w\underset{\sigma\nearrow1}{\longrightarrow}Dw\quad\text{ in }\bs L^\infty(\R^N),\ \text{ for all }w\in\C^1_c(\R^N).
\end{equation}
\end{proposition}
\begin{proof} In Proposition 2.10 of \cite{Kurokawa1981} it is proved that $I_\alpha*h(x)\underset{\alpha\rightarrow0}{\longrightarrow}h(x)$ at each point of continuity of any function $h\in L^p(\R^N)$, $1\le p<\infty$, and it is not difficult to check that this convergence is uniform in $x\in\R^N$ for bounded and uniformly continuous functions (see \cite{ARS2021}). Then \eqref{4.6} is an immediate consequence of Theorem 1.2 of \cite{ShiehSpector2015}, which established that $D^sw=I_{1-s}*Dw$ for all $w\in\C^\infty_c(\R^N$), being the proof equally valid for functions only in $\C^1_c(\R^N)$.
\end{proof}

\begin{remark} The convergence \eqref{4.6}, as well as in $\bs L^p(\R^N)$ for $p\ge 1$, has been shown in Proposition 4.4 of \cite{ComiStefani2019} for functions of $\C^2_c(\R^N)$. By density of $\C^\infty_c(\R^N)$ in $L^p(\R^N)$ for $p\ge 1$, in \cite{BellidoCuetoMora2021} it was shown that the convergence $D^\sigma h\underset{\sigma\nearrow1}{\longrightarrow}Dh$ holds in $L^p(\R^N)$, for $1<p<\infty$, if $h\in W^{1,p}(\R^N)$.
\end{remark}

For $\chi\in L^\infty(\R^N)'$, we denote its restriction to $\Omega	\subset\R^N$ by $\chi_\Omega\in L^\infty(\Omega)'$, defined by
$$\langle\chi_\Omega,\varphi\rangle=\langle\chi,\widetilde\varphi\rangle,\quad\forall\varphi\in L^\infty(\Omega),$$
where $\widetilde\varphi$ is the extension of $\varphi$ by zero to $\R^N\setminus\Omega$,

\begin{theorem}\label{thm4.3} Let $\bs f\in\bs L^1(\R^N)$ ($f_\#=0$) and let $g$ be given as in \eqref{3.1}.
Then, if $(u^\sigma, \lambda^\sigma)\in \Upsilon_\infty^\sigma(\Omega)\times L^\infty(\R^N)'$ are the solutions to \eqref{3.11}-\eqref{3.12}, we have, for a generalised subsequence, the convergences, for any $s$, $0<s<\sigma<1$:
\begin{equation}\label{cov_sigma}
u^\sigma\underset{\sigma\nearrow1}{\longrightarrow}u\quad\text{ in } H^s_0(\Omega)\quad\text{ and }\quad \lambda_{\Omega}^\sigma\underset{\sigma\nearrow1}{\lraup}\lambda\quad\text { in }L^\infty(\Omega)'\text{-weakly}^*,
\end{equation}
where $(u,\lambda)\in W^{1,\infty}_0(\Omega)\times L^\infty(\Omega)'$ is a solution to \eqref{4.3}-\eqref{4.4} and $u$ is the unique solution to \eqref{4.1}-\eqref{4.2}.
\end{theorem}
\begin{proof} Setting $v=0$ in \eqref{2.9}, or $w=u^\sigma$ in \eqref{3.11}, we immediatly obtain
\begin{equation} \label{4.8}\|u^\sigma\|_{H^\sigma_0(\Omega)}=\|D^\sigma u^\sigma\|_{\bs L^2(\R^N)}\le\Big(\frac{g^*}{a_*}\|\bs f\|_{L^1(\R^N)}\Big)^\frac12\equiv C_1,
\end{equation}
where $C_1$ is independent of $\sigma$, $0<\sigma<1$. Hence, arguing as in \eqref{3.24extra}, using \eqref{3.11}-\eqref{3.12}, it also follows easily
\begin{equation}\label{4.9}
\|\lambda^\sigma\|_{L^\infty(\R^N)'}=\sup_{\begin{array}{c}
\varphi\in L^\infty(\R^N)\\
\|\varphi\|_{L^\infty(\R^N)}=1
\end{array}}\langle\lambda^\sigma,\varphi\rangle\le\frac{\|\bs f\|_{\bs L^1(\R^N)}}{g_*^2}.
\end{equation}
Then, using $\Lambda^\sigma=\lambda^\sigma D^\sigma u^\sigma$ and recalling $\|D^\sigma u^\sigma\|_{\bs L^\infty(\R^N)}\le g^*$, from the estimates \eqref{4.8} and \eqref{4.9}, we may take a generalised subsequence $\sigma\nearrow1$ such that, by the compactness of $H^\sigma_0(\Omega)\hookrightarrow H^s_0(\Omega)$, $0<s<\sigma\le 1$,
\begin{align}\label{4.10}
&\begin{cases}u^\sigma\underset{\sigma\nearrow1}{\longrightarrow}u\quad\text{ in }H^s_0(\Omega),\\
D^\sigma u^\sigma
\underset{\sigma\nearrow1}{\lraup} \bs\chi\quad\text{ in }\bs L^2(\R^N)'\text{-weak and }\bs L^\infty(\R^N)'\text{-weak}^*,\end{cases}
\end{align}
\begin{equation}\label{4.11}
\lambda^\sigma\underset{\sigma\nearrow1}{\lraup}\widetilde\lambda\quad\text{ in }L^\infty(\R^N)'\text{-weak}^*,\qquad \Lambda^\sigma\underset{\sigma\nearrow1}{\lraup}\widetilde\Lambda\quad\text{ in }\bs L^\infty(\R^N)'\text{-weak} .
\end{equation}
	
Denoting by $\widetilde u^\sigma$ the extension of $u^\sigma$ by zero to $\R^N\setminus\Omega$, from \eqref{4.10} we conclude that $\bs\chi=D\widetilde u $  and in fact $u\in H^1_0(\Omega)$, and then $D\widetilde u=\widetilde{Du}$. Indeed, recalling the convergence \eqref{4.6}, we have
\begin{equation*}\int_{\R^N}\bs\chi\cdot\bs\varphi=\lim_{\sigma\nearrow1}\int_{\R^N}D^\sigma u^\sigma\cdot\bs\varphi=-\lim_{\sigma\nearrow1}\int_{\R^N}\widetilde u^\sigma(D^\sigma\cdot\bs\varphi)
=-\int_{\R^N}\widetilde u(D\cdot\bs\varphi)=\int_{\R^N}D\widetilde u\cdot\bs\varphi,
\end{equation*}
with an arbitrary $\bs\varphi\in\bs \C^\infty_c(\R^N)$.

On the other hand, given any measurable set $\omega\subset\Omega$, we have now
$$\int_\omega|Du|^2\le\varliminf_{\sigma\nearrow1}\int_\omega|D^\sigma u^\sigma|^2\le\int_\omega g^2$$
and therefore $|Du|\le g$ a.e. in $\Omega$, which yields $u\in\K_g \subset W^{1,\infty}_0(\Omega)$.

Passing to the limit $\sigma\nearrow1$ in \eqref{3.11}, first with  $w\in\C^\infty_c(\Omega)$
$$\int_{\R^N}AD^\sigma u^\sigma\cdot D^\sigma w+\bsl\Lambda^\sigma,D^\sigma w\bsr=\int_{\R^N}\bs f\cdot D^\sigma w$$
and using \eqref{4.6}, \eqref{4.10} and \eqref{4.11}, since $\bs\chi=\widetilde{Du}$ and $D\widetilde w=\widetilde{Dw}$, we obtain
\begin{equation}\label{4.13}
\int_\Omega ADu\cdot Dw+\bsl\Lambda,Dw\bsr=\int_{\Omega}\bs f\cdot Dw,
\end{equation}
by setting $\Lambda=\widetilde\Lambda_{\Omega}$ and $\bsl\Lambda,Dw\bsr=\bsl\widetilde\Lambda,D\widetilde w\bsr$.

Noting that for each $w\in W^{1,\infty}_0(\Omega)$ we may choose $w_\nu\in \C^\infty_c(\Omega)$ such that $w_\nu\underset{\nu\rightarrow\infty}{\longrightarrow}w$ in $H^1_0(\Omega)$ and $Dw_\nu\underset{\nu\rightarrow\infty}{\lraup}Dw$ in $\bs L^\infty(\Omega)$-weak$^*$, in \eqref{4.13} and we may pass to the generalised limit $\nu\rightarrow\infty$, concluding that \eqref{4.13} also holds for all $w\in W^{1,\infty}_0(\Omega)$. So, in order to see that $u$ and $\lambda=\widetilde\lambda_{|_{\Omega}}$, i.e. the restriction to $\Omega$ of the limit charge $\widetilde\lambda$ in \eqref{4.11}, solve \eqref{4.3}, we need to show that
\begin{equation}\label{4.14}
\bsl\Lambda,D w\bsr=\bsl\lambda D u,Dw\bsr=\langle\lambda, Du\cdot Dw\rangle,\qquad\forall w\in W^{1,\infty}_0(\Omega).
\end{equation}

We show first \eqref{4.14} for $w=u$, i.e. $\bsl\Lambda,Du\bsr=\langle\lambda, |Du|^2\rangle$, in two steps. 

Observing that $\widetilde\lambda\ge0$ and $|Du|\le g$, we have $\langle\lambda,|Du|^2	\rangle\le \bsl\Lambda,Du\bsr$ from
\begin{align}\label{4.15}
\nonumber\langle\lambda,|Du|^2\rangle\le\langle\widetilde\lambda,g^2\rangle&=\lim_{\sigma\nearrow1}\langle \lambda^\sigma,g^2\rangle=\lim_{\sigma\nearrow1}\langle \lambda^\sigma,|D^\sigma u^\sigma|^2\rangle\\
&=\lim_{\sigma\nearrow1}\bsl \lambda^\sigma D^\sigma u^\sigma,D^\sigma u^\sigma\bsr\\
\nonumber &=\varlimsup_{\sigma\nearrow1}\int_{\R^N}(\bs f-AD^\sigma u^\sigma)\cdot D^\sigma u^\sigma\\
\nonumber&\le\int_{\R^N}(\bs f-AD\widetilde u)\cdot D\widetilde u=\bsl\widetilde\Lambda,D\widetilde u\bsr=\bsl\Lambda,Du\bsr.
\end{align}

Note that $D^\sigma u^\sigma\underset{\sigma\nearrow1}{\lraup}D\widetilde u$ in $\bs L^2(\R^N)$-weak and hence
$$\varliminf_{\sigma\nearrow1}\int_{\R^N}AD^\sigma u^\sigma\cdot D^\sigma u^\sigma\ge\int_{\R^N} AD\widetilde u\cdot D\widetilde u=\int_\Omega ADu\cdot Du.$$

On the other hand, we find $\bsl\Lambda,Du\bsr\le\langle\lambda,|Du|^2\rangle$ by noting that $\Lambda^\sigma\!=\!\lambda^\sigma D^\sigma u^\sigma$ and, similarly,
\begin{equation}\label{4.16}
0\le\langle\lambda^\sigma,|D^\sigma u^\sigma-D\widetilde u|^2\rangle=\bsl\lambda^\sigma D^\sigma u^\sigma,D^\sigma u^\sigma\bsr-2\bsl\Lambda^\sigma,D\widetilde u\bsr+\langle\lambda^\sigma,|D\widetilde u|^2\rangle
\end{equation}
yields
\begin{align*}
2\bsl\widetilde\Lambda,D\widetilde u\bsr&=2\lim_{\sigma\nearrow1}\bsl\Lambda^\sigma,D\widetilde u\bsr\le\varlimsup_{\sigma\nearrow1}\int_{\R^N}(\bs f-AD^\sigma u^\sigma)\cdot D^\sigma u^\sigma+\lim_{\sigma\nearrow1}\langle\lambda^\sigma,|D\widetilde u|^2\rangle\\
&\le\int_{\R^N}(\bs f-AD\widetilde u)\cdot D\widetilde u+\langle\lambda,|Du|^2\rangle=\bsl\widetilde \Lambda,D\widetilde u\bsr+\langle\lambda,|Du|^2\rangle.
\end{align*}

As a consequence of $\bsl\Lambda,D u\bsr=\langle\lambda,|Du|^2\rangle$, from \eqref{4.16} we deduce
\begin{equation}\label{4.16extra}
\lim_{\sigma\nearrow1}\langle\lambda^\sigma,|D^\sigma u^\sigma -D\widetilde u|^2\rangle=0,
\end{equation}
which, by H\"older inequality yields for any $\bs\beta\in\bs L^\infty(\R^N)$
\begin{align*}
\left|\bsl\widetilde\Lambda-\widetilde\lambda D\widetilde u,\bs \beta\bsr\right|&=\lim_{\sigma\nearrow1}\left|\bsl\Lambda^\sigma-\lambda^\sigma D\widetilde u,\bs\beta\bsr\right|=\lim_{\sigma\nearrow1}\left|\bsl\lambda^\sigma(D^\sigma u^\sigma-D\widetilde u),\bs\beta\bsr\right|\\
&\le\lim_{\sigma\nearrow1}\langle\lambda^\sigma,|D^\sigma u^\sigma-D\widetilde u|\,|\bs\beta|\rangle\\
&\le\lim_{\sigma\nearrow1}\langle\lambda^\sigma,|D^\sigma u^\sigma-D\widetilde u|^2\rangle^\frac12\langle\lambda^\sigma,|\bs\beta|^2\rangle^\frac12=0,
\end{align*}
and, consequently, \eqref{4.14} follows from
\begin{equation*}
\Lambda=\lambda Du\quad\text{ in }\bs L^\infty(\Omega)'.
\end{equation*}

This equality in \eqref{4.14} with $g>0$ implies that
$$\langle\lambda,|Du|^2\rangle=\langle\widetilde \lambda,g^2\rangle\ge\langle\lambda,g_{|_{\Omega}}^2\rangle\ge\langle\lambda,|Du|^2\rangle,$$
and $\langle\lambda,|Du|^2-g^2\rangle=0$ (here $g=g_{|_{\Omega}}$). Then, exactly the same argument as at the end of the proof of Theorem \ref{3.4} shows that $\lambda$ and $u$ satisfy the third condition of \eqref{4.4}.

Finally, since we also have
$$\bsl\lambda Du,D(v-u)\bsr\le 0,\quad\forall v\in\K _g,$$
\eqref{4.3} implies \eqref{4.2} and this concludes the proof of the theorem. \end{proof}

\begin{remark}
In the Hilbertian case of $g\in L^2(\Omega), g\ge0$ and $\bs f\in\bs L^2(\R^N)$, it is easy to show
the convergence of the solutions $(u^\sigma,\gamma^\sigma)\in\Upsilon^\sigma_\infty(\Omega)\times H^\sigma_0(\Omega))$ given by Theorem \ref{2.1}, also in the case $f_\#=0$ to simplify, as $\sigma\nearrow1$ to the local problem for $(u,\gamma)\in W^{1,\infty}_0(\Omega)\times H^1_0(\Omega)$, satisfying \eqref{2.11} with $\sigma=1$ and
\begin{equation}\label{4.19}
\int_\Omega\big(ADu+D\sigma)\cdot Dv=\int_\Omega\bs f\cdot Dv,\qquad\forall v\in H^1_0(\Omega).
\end{equation}
Indeed, as in \eqref{2.10} and \eqref{2.13}, the a priori estimates
$$\|u^\sigma\|_{H^\sigma_0(\Omega)}\le\tfrac1{a_*}\|\bs f\|_{\bs L^2(\R^N)}\quad\text{ and }\quad\|\gamma^\sigma\|_{H^\sigma_0(\Omega)}\le\big(1+\tfrac{a^*}{a_*}\big)\|\bs f\|_{\bs L^2(\R^N)}$$
allows us to take sequences
$$u^\sigma\underset{\sigma\nearrow1}{\longrightarrow}u\quad\text{ and }\quad \gamma^\sigma\underset{\sigma\nearrow1}{\longrightarrow}\gamma\quad\text{ in }H^s_0(\Omega),\quad 0<s<1,$$
in \eqref{2.12} with $v\in H^1_0(\Omega)\subset H^\sigma_0(\Omega)$, in order to obtain \eqref{4.19} and, using \eqref{2.19}, the $\Gamma=\Gamma(u)\in H^{-\sigma}(\Omega)$ corresponding to $\gamma$ satisfies \eqref{2.11} with $\sigma=1$.
\end{remark}

\section*{Ackowledgment} 

\noindent The research of Jos\'e-Francisco Rodrigues was partially done under
the framework of the Project PTDC/MATPUR/28686/2017 at CMAFcIO/ULisboa.

\vspace{3mm}

\noindent The  research  of  Assis Azevedo and Lisa Santos   was  partially  financed  by  Portuguese  Funds 
through  FCT  (Funda\c c\~ao  para  a  Ci\^encia  e  a  Tecnologia)  within  the  Projects  UIDB/00013/2020  and 
UIDP/00013/2020 

\bibliographystyle{emsplain}

\end{document}